\documentclass[11pt]{article}
\usepackage[utf8]{inputenc}
\usepackage[english]{babel}
\usepackage{hyperref}
\usepackage{titling}
\usepackage{amsmath}
\usepackage{amsthm,amssymb}
\newtheorem{theorem}{Theorem}
\newtheorem{lemma}{Lemma}
\newtheorem{corollary}{Corollary}

\usepackage{mathrsfs}
\usepackage{pgfplots}
\usepackage{bbm}
\usepackage{amsfonts}
\usepackage{latexsym}
\usepackage{enumerate}
\usepackage{wasysym}
\usepackage{lipsum}
\usepackage [english]{babel}
\usepackage [autostyle, english = american]{csquotes}
\MakeOuterQuote{"}
\newcommand{\interior}[1]{%
  {\kern0pt#1}^{\mathrm{o}}%
}

\begin{document}

\title{Operator-Theoretical Treatment of Ergodic Theorem and Its Application to Dynamic Models in Economics}

\author{Shizhou Xu}\maketitle

\begin{abstract}
The purpose of this paper is to study the time average behavior of Markov chains with transition probabilities being kernels of completely continuous operators, and therefore to provide a sufficient condition for a class of Markov chains that are frequently used in dynamic economic models to be ergodic. The paper reviews the time average convergence of the quasi-weakly complete continuity Markov operators to a unique projection operator. Also, it shows that a further assumption of quasi-strongly complete continuity reduces the dependence of the unique invariant measure on its corresponding initial distribution through ergodic decomposition, and therefore guarantees the Markov chain to be ergodic up to multiplication of constant coefficients. Moreover, a sufficient and verifiable condition is provided for the ergodicity in economic state Markov chains that are induced by exogenous random shocks and a correspondence between the exogenous space and the state space.
\end{abstract}

\section{Introduction}

Equilibria of dynamic models has been of great interest and importance in Economics. In this paper, we first study the ergodic behavior in certain Markov operators, then provide conditions to approach and eventually guarantee the ergodicity in the Markov operators that are induced by an exogenous Markov process and a measurable selection of the upper hemicontinuous correspondence between the exogenous space and the economic state space.\\

Due to Grandmont \cite{Grandmont}, it is possible to describe the temporary state of an economy at a given time period by an exogenous variable $e \in E$ and an endogenous variable $d \in D$. Therefore, the state space of the economy could be represented by $X = E \times D$ and the state of the economy at time $t$ is represented by $x_t = (e_{t},d_{t})$. We assume that all the randomnesses of the economic evolution are captured by the exogenous shocks. Therefore, $(e_t)_{t \geq 0}$ is a stochastic process on some probability space $(\Omega,\mathcal{F},P)$ with state space $(E,\mathcal{E},Pe_t^{-1})$, where $e_t: \Omega \rightarrow E$ is an exogenous random variable. On the other hand, it follows naturally from our assumption that the endogenous variables are the direct consequence of the activities of the agents in the economy and hence are fully deterministic. Since agents make decisions based on the current exogenous shock and the historical economic states, it is reasonable to assume the existence of some evolution laws to describe $d_t$ based on $e_t$ and $(x_s)_{s \leq t-1}$. Since $x_t = (e_t,d_t)$, we further assume the the existence of evolution laws to describe $x_t$ based on $e_t$ and $(x_s)_{s \leq t-1}$.\\

Now, if the evolution of the economy is Markovian, we then have that $x_t$ depends merely on $e_{t}$ and $x_{t-1}$. Let the evolution of the economy be described by the correspondence $F: X \times E \rightarrow E$ which is assumed to be upper hemi-continuous, and assume the law of evolution $f: X \times E \rightarrow E$ is a measurable selection from $F$, we have $x_{t} = f(x_{t-1},e_{t})$. Such a measurable selection is possible due to the measurable selection theorem of Kuratowski and Ryll-Nardzewski \cite{Nardzewski}.\\

Moreover, let $q: E \times \mathcal{E} \rightarrow [0,1]$ be the transition probability of the Markov process $(e_t)_{t \geq 0}$, and let $Q$ denote the set of all exogenous transition probabilities. For each measurable selection $f$ from the correspondence $F$, we define an induced transition probability $p: X \times \mathcal{X} \rightarrow [0,1]$ of the Markov process $(x_t)_{t \geq 0}$ by $p(x,A) := q(e,\{e \in E: f(x,e) \in A\})$ for all $A \in \mathcal{X}$. Let $P$ denote the set of all such induced transition probabilities on the state space. Moreover, we can represent $p(x,A) = q(\pi_i(x),f_x^{-1}(A))$ at each time period $i$, where $\pi_i: X \rightarrow E$ is defined by $\pi_i(x) := \bigcup_{x_{i-1}}\{e \in E: f(x_{i-1},e) = x\}$ for each $i \in \mathbb{N}$, and $f_x^{-1}: \mathcal{X} \rightarrow \mathcal{E}$ by $f_x^{-1}(A) = \{e \in E: f(x,e) \in A\}$ for each $x \in X$. The map $f_x^{-1}: \mathcal{X} \rightarrow \mathcal{E}$ is continuous due to Blume \cite{Blume} and the $p(x,A)$ is a probability transition probability on $(X,\mathcal{X})$ due to Jean-Michel Grandmont and Werner Hildenbrand \cite{Hildenbrand}. It remains to define the operators that are induced by $p$.\\

Let $(M)$ denote the Banach space of all complex-valued countably additive set functions on some measurable space $(X,\mathcal{X},||\cdot||_X)$ equipped with the variation norm. Given each fixed Markov transition probability $p$ on $(X,\mathcal{X},||\cdot||_X)$, we define the induced operator $T : (M) \rightarrow (M)$ by $T\mu(E) := \int_E \mu(dt)p(t,E), \forall \mu \in (M), \forall E \in \mathcal{X}$. Also, define $T^i\mu := T(T^{i-1}\mu)$ inductively for each $i \in \mathbb{N}$ with $T^0 = I$. $T\mu(E)$ can be interpreted as the probability of event $E$ if the initial distribution of $(X,\mathcal{X})$ is $\mu$.\\

Let $M^*$ denote the Banach space of all complex-valued bounded Borel measurable functions on $(X,\mathcal{X},||\cdot||_X)$ equipped with the supreme norm. Given a fixed Markov transition probability $p$ on $(X,\mathcal{X},||\cdot||_X)$, we define the induced operator $T^*$ by $T^*l(x) := \int_X l(t)p(x,dt), \forall l \in M^*, \forall x \in X$. Also, define $(T^*)^il := T^*((T^*)^{i-1}l)$ inductively for each $i \in \mathbb{N}$ with $(T^*)^0 = I$. $T^*l(x)$ can be interpreted as the expectation of $l$ if starting at $x \in X$.\\

We call the operator $T$ $Markov$ $operator$ since it is induced by a Markov transition probability $p$, and we say $p$ is the $kernel$ of $T$. Also, we call $T^*$ the $dual$ of $T$. It is not difficult to see that there is a one-to-one relationship between a Markov operator and its kernel. Also, it is not hard to see that $T$ and and its dual are connected by $\int_X \mu(dt)T^*l(t) = \int_X \mu(dt) (\int_X p(t,ds)l(s)) = \int_X l(s) (\int_X \mu(dt)p(t,ds)) = \int_X T\mu(ds)l(s), \forall \mu \in (M), l \in M^*$.\\

We achieve two main results in this paper: The first is an explicit the ergodic decomposition formula of the invariant distribution for each given initial distribution under a quasi-strongly completely continuous Markov operator. The explicit formula shows an ergodicity up to a multiplicative constant in the restriction of $p$ to each ergodic parts on the state space if the induced $p$ is the kernel of a quasi-strongly completely continuous Markov operator. We state it in Theorem 1 and Corollary 1. The second main result is a sufficient condition on the correspondence $F$ and the exogenous transition probability $q$ such that every measurable selection $f$ from $F$ together with a $q$ on $E$ will induce an ergodic transition probability $p$ on $X$. That is, the time average of the Markov process based on $p$ converges uniformly to the space average with respect to a unique invariant probability distribution $\mu^*$ on the state space, where $\mu^*$ is independent of the initial distribution. Or, equivalently, $\exists! \mu^*$ such that $\forall x \in X, \lim_{n \rightarrow \infty} \frac{1}{n} \sum_{i=0}^{n-1}(T^*)^ig(x) = \int_X g(x) \mu^*(dx)$ for all bounded linear functionals $g$ on $X$. We state it in Theorem 2 and Corollary 2.\\

In section 2, we first show the uniform ergodicity of quasi-strongly completely continuous operators by following the pioneer work from from Yosida and Kakutani, "Operator-Theoretical Treatment of Markoff's Process and Mean Ergodic Theorem."\cite{Kakutani} In section 3, we apply the uniform ergodic theorem to the general dynamic economic model to drive and prove the two main results of this paper. Finally, in part 4, we conclude this paper and show further research questions.\\

\section{Uniform Ergodic Theorem}

Given a Markov process, the behavior of the repeated operations and the time average has long been interesting to Mathematicians. Neumann studied the classic Mean Ergodic Theorem. Bogolyubov studied the ergodicity of operators under the assumption of quasi-strongly complete continuity. Deoblin studied the ergodicity of Markov transition probability induced transition under certain assumptions. Yosida and Kakutani found that Deoblin's condition actually implies the quasi-strongly complete continuity. In this section, we study the classic Uniform Ergodic Theorem and apply the theorem to the semi-group of operators that are induced by Markov transition probabilities.\\

\subsection{(Quasi-) Weakly Completely Continuity}

Given a Banach space $(B,||\cdot||_B)$, a bounded linear operator $T : B \rightarrow B$ is called $weakly$ $completely$ $continuous$ if the image of the unit ball under $T$ is weakly compact in $B$, i.e. $T(\{\mu \in B: ||\mu||_B \leq 1\})$ is weakly compact. A bounded linear operator $T : B \rightarrow B$ is called $quasi-weakly$ $completely$ $continuous$ if there exists a finite $n \in \mathbb{N}$ and weakly completely continuous operator $V$ such that $||T^n-V||_{op} < 1$.\\

In 1941, Yosida and Kakutani showed that, under a boundedness assumption, a quasi-weakly continuous operator has convergent time average. We state it as Lemma 1.\\

\begin{lemma}
If $T:B \rightarrow B$ is a quasi-weakly completely continuous operator satisfying that $\exists C, < \infty, ||T^i||_{op} \leq C$ for all $i \in \mathbb{N}$, then $\forall \mu \in B$, the sequence $\{\frac{1}{n}\sum_{i=1}^n T^i \mu\}_{n \in \mathbb{N}}$ converges strongly to a $\bar{\mu}$ and the map $T_1: B \rightarrow E_1(T)$ by $T_1(\mu) = \bar{\mu}$ is well-defined and satisfies: $TT_1 = T_1T = T_1^2 = T_1$ and $||T_1||_{op} \leq C$.
\end{lemma}

\begin{proof}
See "Operator-Theoretical Treatment of Markoff's Process and Mean Ergodic Theorem" by Yosida and Kakutani. \cite{Kakutani}
\end{proof}

Due to the property $TT_1 = T_1T = T_1^2 = T_1$, we have $T_1 : B \rightarrow E_1(T)$ is a projection from the Banach space to the eigen space of $T$ with respect to eigen value 1. In other word, for all $\mu \in E_1(T)$, we have $T\mu = \mu$.\\

Lemma 1 is powerful in application because the assumptions can be easily satisfied. In particular, let's again consider the operators that are induced by a Markov transition probability. Since $||T^n||_{op} \leq 1, \forall n \in \mathbbm{N}$, we have the boundedness assumption is always satisfied for all Markov operators. Furthermore, if $B$ is $L^p, 1 < p < \infty$, then condition 2 is also satisfied due to the locally weak compactness of $L^p, \forall p \text{ such that } 1 < p < \infty$. But since $L^1$ is not a locally weakly compact space, we have to impose more condition in order to have condition 2 satisfied in $L^1$ to serve our purpose.\\

\begin{lemma}
If $\exists f(x,t)$ such that  $T^*l(x) := \int_Xl(t)p(x,dt) = \int_X l(t)f(x,t)dt$, then $T$ is weakly completely continuous if and only if $f$ is uniformly integrable, i.e. $\forall \epsilon > 0, \exists \sigma >0$ such that $\forall E \in \mathcal{X}$ satisfying $||E||_X < \sigma$, we have $\int_E  f(x,t)dt < \epsilon, \forall x \in X$.
\end{lemma}

\begin{proof}
See "Stochastic Processes with an Integral-Valued Parameter" by Doob \cite{Doob}.
\end{proof}

Therefore, given a dynamic model in Economics, if we are only interested in the existence of a unique temporary equilibrium that depends on the initial probability distribution on the economic state space, a condition that guarantees the quasi-weakly complete continuity of the induced $T$ is sufficient. And, Lemma 2 has provided an equivalent condition that is more verifiable than the weakly compactness condition that is not straight-forward to verify.\\

\subsection{(Quasi-) Strongly Completely Continuity}

Now, in order to further approach the existence of $\mu^*$ and derive the convergence of $\lim_{n \rightarrow \infty} \frac{1}{n} \sum_{i=0}^{n-1}(T^*)^il(x) = \int_X l(x) \mu^*(dx), \forall x \in X$, for all bounded linear functionals $l$ on $X$, we need to show the uniform convergence of $\{\frac{1}{n} \sum_{i=1}^{n}T^i\}_{n \in \mathbb{N}}$ to implies the uniform (in both $x \in X$ and $E \in \mathcal{X}$ convergence of $\{ \frac{1}{n}\sum_{i=1}^{n}p^{(i)}(x,E) \}_{n \in\mathbb{N}}$, where $p^{(i)}$ is the kernel of $T^i$ for each $i \in \mathbb{N}$. To that end, we need to introduce strongly completely continuity and quasi-strongly completely continuity.\\

A bounded linear operator $T: B \rightarrow B$ is called $strongly$ $completely$ $continuous$ if it maps the unit ball in $B$ to a strongly compact set in $B$, i.e. $T(\{\mu \in B: ||\mu||_B \leq 1\})$ is strongly compact. And, a bounded linear operator $T : B \rightarrow B$ is called $quasi-strongly$ $completely$ $continuous$ if $\exists n \in \mathbb{N}, n < \infty$ and a strongly completely continuous operator $V$ such that $||T^n - V||_{op} < 1$.\\

Since it is clear that quasi-strongly complete continuity implies quasi-weakly complete continuity, we have the same results in Lemma 1. if a quasi-strongly completely continuous $T$ satisfies $||T^i||_{op} < C, \forall i \in \mathbb{N}$. In fact, Yosida and Kakutani have shown that the strong convergence in the the results of Lemma 1 are uniform convergence when assuming quasi-strongly complete continuity of $T$.\\

\begin{lemma}
If $T : B \rightarrow B$ is a quasi-strongly completely continuous linear operator that satisfies $\exists C < \infty, \forall n \in \mathbb{N}, ||T^n||_{op} < C$, then we have:\\
1. $\dim(\sigma(T) \cap \{\lambda \in \mathbb{C}: |\lambda| = 1\}) = k < \infty$;\\
2. Let $\{\lambda_i\}_{i = 1}^k := \sigma(T) \cap \{\lambda \in \mathbb{C}: |\lambda| = 1\}$, then $\forall i \leq k, \dim(E_{\lambda_i}(T)) < \infty$;\\
3. $\forall i \leq k, \exists T_{\lambda_i}$ s.t. $TT_{\lambda_i} = T_{\lambda_i}T = \lambda_iT$;\\
4. $T_{\lambda_i}T_{\lambda_j} = T_{\lambda_j}T_{\lambda_i} = T_{\lambda_i}$ or $0$ if $i = j$ or otherwise, and $||T_{\lambda_i}||_{op} < C$;\\
5. Let $S := T - \sum_{i = 1}^kT_{\lambda_i}$, $T^n = S^n + \sum_{i = 1}^k \lambda_i^n T_{\lambda_i}$, $ST_{\lambda_i} = T_{\lambda_i}S = 0, \forall i$;\\
6. $\exists M, \epsilon > 0$ s.t. $||S^n||_{op} \leq \frac{M}{(1+\epsilon)^n}$.
\end{lemma}

\begin{proof}
See "Operator-Theoretical Treatment of Markoff's Process and Mean Ergodic Theorem" by Yosida and Kakutani. \cite{Kakutani}
\end{proof}

Since the series $\sum_{i = 1}^\infty z^i$ converges for all complex number on the unit circle except 1, we have from Theorem 2 that $\forall \lambda \in \{\lambda \in \mathbb{C}: |\lambda| = 1\}, \exists M > 0$ such that $||\frac{1}{n}(\sum_{i = 1}^n (\frac{T}{\lambda})^i) - T_{\lambda}||_{op} \leq \frac{M}{n}$ and therefore $||\frac{1}{n}(\sum_{i = 1}^n T^i - T_1)||_{op} \leq \frac{M}{n}$. That is, the time average of the Markov operator induced by the Markov transition probability converges in operator norm to a unique operator that is a projection mapping the Banach space to the eigen space of the induced operator corresponding to eigen value 1.\\

\section{Application of Uniform Ergodic Theorem}

In this part, we will apply our results in part 2 to the operators that are induced by the Markov transition probabilities.\\

Our goal is to use the results in part 2 to approach a sufficient condition on the Markov transition probability $q$ on $(E,\mathcal{E})$ and the upper hemicontinuous correspondence $F : E \rightarrow 2^X$ for the induced Markov transition probability $p$ on the economic state space  $(X,\mathcal{X})$ to be ergodic, i.e. $\exists! \mu^*$ such that $\frac{1}{n}\sum_{i=0}^{n-1}T^i\mu$ converges to $\mu^*$ for all probability distribution $\mu$.\\

To this end, we first show that the uniform convergence of $\frac{1}{n} \sum_{i = 1}^{n} T^i$ to $T_1$ implies the uniform (in both $x \in X$ and $E \in \mathcal{X}$) convergence of $\frac{1}{n} \sum_{i = 1}^{n} p^{(i)}$ to $p_1$, the the kernel of $T_1$ in the following lemma.\\

\begin{lemma}
If a Markov operator $T: (M) \rightarrow (M)$ satisfies $\frac{1}{n} \sum_{i = 1}^{n} T^i$ converges uniformly to a operator $T_1$, then $\frac{1}{n} \sum_{i = 1}^{n} p^{(i)}$ converges uniformly (in both $x \in X$ and $E \in \mathcal{X}$) to a Markov transition probability $p_1$, and $p_1$ is the kernel of $T_1$.
\end{lemma}

\begin{proof}
Since the uniform convergence in operators space is equivalent to operator norm convergence, but operator norm convergence in Markov operators is also equivalent to the supreme norm convergence in the corresponding kernels, we are done due to the one-to-one correspondence between a Markov operator and its kernel.
\end{proof}

After acquiring the result in Lemma 4, we have the following corollary from Lemma 3 and Lemma 4. We state it as Lemma 5.\\

\begin{lemma}
If $T$ is quasi-strongly completely continuous, then for each $n \in \mathbb{N}$, we have $p^{(n)}(x,E) = \sum_{i=1}^k \lambda_i^np_{\lambda_i}(x,E) + S^{(n)}(x,E)$, where:\\
1. $k = \dim(\sigma(T)\cap \{\lambda \in \mathbb{C}: |\lambda| = 1\}) < \infty$;\\
2. $\sup\limits_{x \in X, A \in \mathcal{X}} |\frac{1}{n}\sum_{i=1}^n \frac{p^{(i)}(x,A)}{\lambda_i} - p_{\lambda_i}(x,A)| \leq \frac{M}{n}$;\\
3. $\int_X p^{(n)}(x,dt)p_{\lambda_i}(t,E) = \int_X p_{\lambda_i}(x,dt)p^{(n)}(t,E) = \lambda_i^np_{\lambda_i}(x,E)$;\\
4. $\int_X p_{\lambda_i}p_{\lambda_j} = \int_X p_{\lambda_j}p_{\lambda_i} = p_{\lambda_i}(x,E)$ or $0$ if $i = j$ or otherwise;\\
5. $\int_X S(x,dt)p_{\lambda_i}(t,E) = \int_X p_{\lambda_i}(x,dt)S(t,E) = 0, \forall i \in \{1,2,...,k\}$;\\
6. $\sup_{x \in X, A \in \mathcal{X}}|S^{(n)}(x,A)| \leq \frac{M}{(1+\epsilon)^n}$;\\
where $M$ and $\epsilon$ are positive constants independent of $n$, and $\int_X p_{\lambda_i}p_{\lambda_j} := \int_X p_{\lambda_i}(x,dt)p_{\lambda_j}(t,E)$.
\end{lemma}

It is clear that Lemma 5 follows entirely from Lemma 3 and Lemma 4. Before stating our main results, it still remains to show the ergodic decomposition of both the state space and the eigen space with respect to eigen value 1, given a quasi-strongly completely continuous Markov operator $T$. Such ergodic decompostions have been proved by Yosida and Kakutani in \cite{Kakutani}. We state the results in Lemma 6.\\

\begin{lemma}
If $T$ is quasi-strongly completely continuous and $p_1$ is the kernel of $T_1$, then $p_1(t,E) = \sum_{\alpha =1}^l y_{\alpha}(t)\nu_{\alpha}(E)$, and $X = \sum_{\alpha = 1}^l E_{\alpha} + \Delta$ such that:\\
1. $\{\nu_{\alpha}\}_{\alpha = 1}^l$ forms a basis for $\{\mu \in (M): \mu \geq 0, ||\mu||_{va} = 1, T\mu = \mu\}$ and $\mu_{\alpha}$ are mutually singular;\\
2. $\{y_{\alpha}\}_{\alpha = 1}^l$ forms a basis for $\{y \in M^*: T^*y = y\}$ satisfying $T^*y_{\alpha} = y_{\alpha}, y_{\alpha} \geq 0$, and $\sum_{\alpha =1}^l y_{\alpha}(x) = 1, \forall x \in X$;\\
3. $\int_X \nu_{\alpha}(dx)y_{\beta}(x) = 1$ or $0$ if $\alpha = \beta$ or otherwise;\\
4. $x_{\alpha}(E_{\beta}) = 1$ or $0$ if $\alpha = \beta$ or otherwise;\\
5. $p(x,E_\alpha) = 1, \forall x \in E_{\alpha}$;\\
6. $\sup\limits_{x \in E_{\alpha}, A \subset E_{\alpha}} |\frac{1}{n}\sum_{i=1}^n p^{(i)}(x,A) - \nu_{\alpha}(E)| \leq \frac{M}{n}$;\\
7. $\sup\limits_{x \in X} \frac{1}{n}\sum_{i=1}^n p^{(i)}(x,\Delta) \leq \frac{M}{n}$;\\
where $M$ is a positive constant independent of $n$, and $\{E_{\alpha}\}_{\alpha=1}^l$ a system of mutually disjoint sets.
\end{lemma}

\begin{proof}
See "Operator-Theoretical Treatment of Markoff's Process and Mean Ergodic Theorem" by Yosida and Kakutani. \cite{Kakutani}
\end{proof}

Finally, we derive an explicit formula for the the ergodic decomposition of the invariant probability measure for each initial distribution, under the assumption of quasi-strongly complete continuity of $T$. We state the result in Theorem 1.\\

\begin{theorem}
If $T: (M) \rightarrow (M)$ has quasi-strongly complete continuity, then for any initial probability measure $\mu$ on $(X,\mathcal{X})$, there exists a system of mutually disjoint sets $\{E_{\alpha}\}_{\alpha = 1}^l$ and a system of mutually singular set functions $\{\nu_{\alpha}\}_{\alpha = 1}^l$ such that $\frac{1}{n} (\sum_{i=0}^{n-1} T^i\nu(E)) \text{ converges uniformly (in } E \in \mathcal{X} \text{) to } \sum_{\alpha = 1}^l \nu(E_{\alpha})\mu_{\alpha}(E \cap E_{\alpha})$.
\end{theorem}

\begin{proof}
By the quasi-strongly complete continuity of $T$, we have: $p^{(n)}(t,E) = \sum_{i=1}^k \lambda_i^np_{\lambda_i}(t,E) + S^{(n)}(t,E)$ from Lemma 5; $\frac{1}{n} (\sum_{i=1}^{n} T^i)$ converges uniformly to $T_1$ and the kernel of $T_1$ is decomposible: $p_1(t,E) = \sum_{\alpha =1}^l y_{\alpha}(t)\nu_{\alpha}(E)$ from Lemma 6; and $X = \sum_{\alpha = 1}^l E_{\alpha} + \Delta$, where $E_{\alpha}$ is defined by $y_{\alpha}^{-1}(\{1\})$ from Lemma 6. Now, let $\mu \in (M)$ be an arbitrary initial probability measure on $X$, then we have:

\begin{equation*}
\begin{split}
\frac{1}{n}(\sum_{i=0}^{n-1} T^i\mu(E)) & = \frac{1}{n}(\sum_{i=0}^{n-1} \int_X \mu(dt)p^{(i)}(t,E)) \\
  & =  \int_X (\sum_{i=1}^{n-1} (p^{(i)}(t,E) \frac{1}{n-1} \frac{n-1}{n} + \frac{1}{n}\mathbbm{1}_E)) \mu(dt)\\
  & = \frac{n-1}{n} \int_X (\frac{1}{n-1} \sum_{i=1}^{n-1} (p^{(i)}(t,E))\mu(dt) + \frac{1}{n}\mu(E)\\
  & = \frac{n-1}{n} \int_X (\frac{1}{n-1} \sum_{i=1}^{n-1} (\sum_{j=1}^k \lambda_j^i p_{\lambda_j}(t,E) + S^{(i)}(t,E))\mu(dt) + \frac{1}{n}\mu(E)\\
  \text{Let } n \text{ goes to infinity } & \rightarrow \int_X p_1(t,E) \mu(dt)\\
  & = \sum_{\alpha =1}^l \int_X  y_{\alpha}(t)\nu_{\alpha}(E) \mu(dt)\\
  & = \sum_{\alpha = 1}^l \sum_{\beta = 1}^l \int_{E_{\beta}} y_{\alpha}(t)\nu_{\alpha}(E \cap E_{\beta}) \mu(dt)\\
  & = \sum_{\alpha =1}^l \int_{E_\alpha} \nu_{\alpha}(E \cap E_{\alpha}) \mu(dt)\\
  & = \sum_{\alpha =1}^l \mu(E_\alpha) \nu_{\alpha}(E \cap E_{\alpha})
\end{split}
\end{equation*}

The second line follows from $p^{(0)}(t,E) = \mathbbm{1}_E$, the forth from Lemma 5, the fifth from Lemma 5 and the fact that $\frac{1}{n}\sum_{i = 1}^n \lambda^i \rightarrow 0$ as $n \rightarrow \infty$ for all $\lambda \in \{\lambda \in \mathbb{C}: |\lambda| = 1, \lambda \neq 1\}$, the sixth and the seventh from Lemma 6, the eighth from the fact that $y_{\alpha}(t) = 1$ if $t \in E_{\alpha}$ and $y_{\alpha}(t) = 0$ otherwise. Therefore, we have finished our proof.
\end{proof}

Notice here the $\{\nu\}_{\alpha = 1}^l$ and $\{E_{\alpha}\}_{\alpha = 1}^l$ depend merely on $T$, and the coefficients $\mu(E_{\alpha})$ are the initial probability measure of the ergodic parts. Moreover, it is not hard to see from Theorem 1 that, if we consider $E_{\alpha}$ as a space itself, then for any $\mu \in \{\mu \in (M): ||\mu||_{va} = 1\}$, we have $\frac{1}{n}\sum_{i = 0}^{n-1} T^i \mu|_{E_{\alpha}}$ converges uniformly (in $E \in \mathcal{X}$) to $\mu(E_{\alpha}) \nu_{\alpha}$. Therefore, we have the following corollary.

\begin{corollary}
If $p$ is the kernel of a quasi-strongly completely continuous operator $T$, let $\{E_{\alpha}\}_{\alpha = 1}^l$ be the ergodic parts, we have $p|_{E_{\alpha}}$ is ergodic for each $\alpha \in \{1,2,...,l\}$. Moreover, let $\mu$ be an initial distribution, then $\forall f \in M^*, \frac{1}{n} \sum_{i = 0}^{n-1} (T^*)^if\mathbbm{1}_{E_{\alpha}}$ converges uniformly (in $x \in X$) to $\mu(E_{\alpha})\int_{E_\alpha} f(x) \nu_{\alpha}(dx)$, $\alpha \in \{1,2,...,l\}$.
\end{corollary}

Now, it is clear from the corollary that if we have a condition on $q$ and $F$ such that the induced transition probability $p$ becomes a kernel of some quasi-strongly completely continuous operator, then the condition guarantees the ergodic decomposition of the economic state space such that the restriction of $p$ on each ergodic part becomes ergodic.\\

Although the invariant probability distribution is still depending on the initial distribution under the assumption of quasi-strong complete continuity, it is clear from the explicit ergodic decomposition formula in Corollary 1 that the dependence on the initial distribution becomes extremely limited as merely the coefficients changes if $\mu$ changes. Moreover, if the quasi-strongly completely continuous operator induced by $p$ has exactly one ergodic part, then we conclude that $p$ is ergodic.\\

In 1979, in the paper "The Ergodic Behavior of Stochastic Processes of Economic Equilibria," Blume provided a condition on $F$ such that, for a large subset $Q$ (actually a dense set), every measurable selection of evolution law $f$ from a $F$ satisfying his condition will induce an ergodic $p$ on the economic state space. However, he did not notice that his condition is actually a special case of Doeblin condition. Also, the conclusion under his condition is not sharp because he did not show a specific condition for $q$ to be in that dense set of Markov transition probability on $(E,\mathcal{E})$. But Blume's conclusion following from the Banach Fixed Point Theorem distinguishes itself by its extreme simplicity.\\

Now, we are ready to state the second main result of this paper, which provides an applicable condition on $q$ and $F$ such that the induced $p$ has ergodicity.\\

\begin{theorem}
If the correspondence $F$ and $q$ satisfies:\\
1. $\exists K \in \mathcal{X}$ s.t. $f(x,e) \in K, \forall e \in E, x \in X$, measurable selection $f$,\\
2. $\exists e^*, x^*, n \in \mathbb{N}$ s.t. $f^{(n)}(x,e^*) = x^*, \forall x \in K$, measurable selection $f$,\\
3. $\exists \epsilon > 0, q(e,\{e^*\}) \geq \epsilon, \forall e \in E$,\\
then the induced $p$ is ergodic.
\end{theorem}

Note Theorem 2 tells us that if: 1. there exists a small set $K$ on the economic state space such that, starting from any point on the state space, the Markov chain will visit $K$ infinitely often; 2. there exists a special exogenous point such that, starting from any point on the small set $K$, the Markov chain will arrive at a fixed point after finite time periods if the special exogenous point happens repeatedly; 3. all the points on the exogenous space are communicative to the special exogenous condition; then there exists a probability distribution $\mu^*$ on the economic state space such that, the time average converges to the space average with respect to $\mu^*$, independent of the initial distribution.\\

\begin{proof}
Assume $F$ and $q$ satisfies the condition in Theorem 2, let $f$ be an arbitrary measurable selection from $F$, and define $f^{(n)}(x,e) := f(f^{(n-1)}(x,e),e)$ by induction. Also, define an operator $V: (M) \rightarrow (M)$ by $V\mu(A):=\int_X{\epsilon}^n\delta_{x^*}(A)\mu(dx)$ for all $\mu \in (M)$. We claim here that $V$ is strongly completely continuous. Indeed, given $B_{(M)}:=\{\mu \in (M): ||\mu||_v \leq 1\}$, we have $V(B_{(M)}) = \{c\delta_{x^*}: c \in [0,{\epsilon}^n]\}$. Since $\{c\delta_{x^*}: c \in [0,{\epsilon}^n]\}$ is isometric isomorphic to $[0,{\epsilon}^n]$, and $[0,{\epsilon}^n]$ is clearly compact, we have $V(B_{(M)})$ is compact in norm. Hence, $V$ is strongly completely continuous by definition.\\

Now, let $p$ be the transition probability that is induced by $f$ and $q$, define $T$ by $T\mu(\cdot) := \int_X p(x,\cdot)\mu(dx)$ on $\mu \in \{\mu \in (M): \mu \geq 0, ||\mu||_v = 1\}$ as usual, we have $(T-V)\mu(\cdot) = \int_X (p(x,\cdot) - {\epsilon}^n\delta_{x^*}(\cdot)) \mu(dx)$ on $(M)$. But for each $A \in \mathcal{X}$, we have $p(x,A) - {\epsilon}^n\delta_{x^*}(A) \leq p(x,X-\{x^*\}) - 0 < 1$ if $x^* \notin A$, and $p(x,A) - {\epsilon}^n\delta_{x^*}(A) \leq p(x,A) - {\epsilon}^n < 1$ if $x^* \in A$. In both cases, we have $(T-V)\mu(\cdot) = \int_X (p(x,\cdot) - {\epsilon}^n\delta_{x^*}(\cdot)) \mu(dx) < \int_X\mathbbm{1}_X(x)\mu(dx) = 1$. Since our choice of $\mu$ is arbitrary, we have shown $||T-V|| < 1$ by definition and $T$ is quasi-strongly completely continuous. Therefore, we have $1 \in \sigma(T)$, $E_1(T) \neq \emptyset$, and $\exists \{\nu_{\alpha}\}_{\alpha =1}^l$ be a mutually singular basis for $E_1(T)$ by Lemma 6. That completes our proof for existence.\\

It remains to show the uniqueness of the invariant measure. Let's first derive a Stopping Markov process $\{y_i\}_{i=0}^{\infty}$ on $(X,\mathcal{X})$ from $\{x_i\}_{i=0}^\infty$ by $y_0 := x_0$, $y_1 := x_{\tau_K}$, and $y_i$ be the i-th member in $\{x_i\}_{i=0}^\infty$ that arrives $K$. Also, define the transition probability for $\{y_i\}_{i=0}^{\infty}$ by:

\begin{equation*}
\begin{split}
p_K(x,E) & = P(y_{i+1} | y_i = x) \\
& = p(x,E) + \int_{K^c}p(x,dy)p(y,E) + \int_{K^c}p(x,dy)\int_{K^c}p(y,dz)p(z,E) + ...
\end{split}
\end{equation*}

\noindent where $E \subset K$. Now, let $\mu^*$ be an invariant probability measure satisfying $\mu^*(\cdot) = \int_X \mu^*(dx)p(x,\cdot)$. Such $\mu^*$ exists since we have shown the quasi-strongly complete continuity of $T$ above. We have $\mu^*$ must satisfy $\mu^*(K) = 1$. Indeed, if $\mu^*(K) < 1$, $\exists E \subset K^c$ such that $\mu^*(E) > 0$. But since $p(x,K) = 1$ for all $x \in X$, we have $\mu^*(E) = \int_X \mu^*(dx)p(x,E) = 0$, a contradiction.\\

Next, we assume that $K = K_1 \cup K_2$, where $K_1,K_2$ are mutually disjoint, and assume W.O.L.G. that $\delta_{x^*}(K_1) = 1$ and $\delta_{x^*}(K_2) = 0$. Then, the fact $p(x,{x^*}) \geq \epsilon > 0$ for all $x \in K$ implies that $P(y_i \in K_2, \forall i \in \mathbb{N}) \neq 1$. That is, $K$ cannot be decomposed into two disjoint ergodic parts for $\{y_i\}_{i=0}^{\infty}$. Therefore, $p_K$ is ergodic on $K$, and $\exists$ a unique invariant probability measure $\mu_K$ on $K$ such that $\lim_{n \rightarrow \infty} \frac{1}{n}\sum_{i = 0}^{n-1} T_K^i(\mu) = \mu_K$, where $T_K$ is the operator induced by $p_K$. But notice that for all $E \subset K$, we have:

\begin{equation*}
\begin{split}
\mu^*(E) & = \int_X \mu^*(dx)p(x,E)\\
& = \int_K \mu^*(dx)p(x,E) + \int_{K^c} \mu^*(dx)p(x,E)\\
& = \int_K \mu^*(dx)p(x,E) + \int_{K^c} \int_{X} \mu^*(dx)p(x,dt)p(t,E)\\
Fubini \rightarrow& = \int_K \mu^*(dx)p(x,E) + \int_{X}  \mu^*(dx) \int_{K^c} p(x,dt)p(t,E)\\
& = \int_K \mu^*(dx)(p(x,E) + \int_{K^c} p(x,dt)p(t,E)) + \int_{K^c}\int_{K^c}\\
& = \int_K \mu^*(dx)(p(x,E) + \int_{K^c} pp + \int_{K^c}\int_{K^c}ppp) + \int_{K^c}\int_{K^c}\int_{K^c}\\
\end{split}
\end{equation*}

\noindent where $\int_{K^c}\int_{K^c} := \int_{K^c}\int_{K^c}\mu^*(dx)p(x,dt)p(t,E)$, $\int_{K^c}pp:= \int_{K^c}p(x,dt)p(t,E)$, and other abbreviations are defined similarly for convenience.\\

Notice that if we continue this process, the first term on the right side becomes $\int_X \mu^*(dx)p_K(x,E)$ by definition, and the second term on the right side is always non-negative. Therefore, we have $\mu^*(E) \geq \int_X \mu^*(dx)p_K(x,E)$. We claim $\mu^*(E) = \int_X \mu^*(dx)p_K(x,E)$. Indeed, if $\mu^*(E) > \int_X \mu^*(dx)p_K(x,E)$, then there exists a set $E \subset K$ such that $\mu^*(E) > \int_X \mu^*(dx)p_K(x,E)$, because $\mu^*(K) = \int_X \mu^*(dx)p_K(x,K) = 1$ as they are both probability measure on $K$. But then we have $\mu^*(K-E) < \int_X \mu^*(dx)p_K(x,K-E)$, a contradiction. Therefore, $\mu^*(E) = \int_X \mu^*(dx)p_K(x,E)$. That is $\mu^*$ is invariant under $T_K$. Since our choice of $\mu^*$ is arbitrary, we conclude that if $\mu^*$ is invariant under $T$, then it is invariant under $T_K$.\\

Finally, Assume there exists another probability measure $\mu^*_1$ that is invariant under $T$, then we must have $\mu^*_1(K) = \mu^*(K) = 1$ and hence $\exists E \subset K$ such that $\mu^*_1(E) \neq \mu^*(E)$. But we also have $\mu^*_1$ being invariant under $T-K$. That implies $\mu^*_1$ is invariant under $T_K$ and therefore $\mu^*_1(E) = \mu^*(E), \forall E \subset K$ by the uniqueness of invariant probability measure under $T_K$. The contradiction shows the uniqueness of invariant probability measure under $T$ and completes our proof.
\end{proof}

Here, we have used quasi-strongly complete continuity to prove the existence, and created a stopping Markov process to show the uniqueness. In fact, there is an easier way to prove Theorem 2 by using the classic Harris positive recurrence condition.\cite{Harris} Before showing an alternative proof, we first state two classic conditions for Markov operators to have a unique invariant probability measure and hence to be ergodic:

\begin{multline}
\tag{D}
\exists \mu \in \{\mu \in (M): \mu \geq 0, ||\mu||_v = 1\} \text{ on } (X,\mathcal{X}), \epsilon > 0,\\ \text{ such that } p(x,E) \geq \epsilon\mu(E), \forall x \in X, \forall E \in \mathcal{X}.
\end{multline}

\begin{multline}
\exists \mu \in \{\mu \in (M): \mu \geq 0, ||\mu||_v = 1\}, K \in \mathcal{X}, k \geq 1, \epsilon > 0 \text{ such that }  \\ \sup_{x \in X} E_x \tau_K < \infty \text{, and } p^{(k)}(x,E) > \epsilon\mu(E), \forall x \in K, \forall E \in \mathcal{X}.
\tag{H}
\end{multline}

Here, (D) is the Doeblin's condition and (H) is the Harris positive recurrence condition. It is not hard to see that Harris condition implies Doeblin's condition if we consider the entire space $X$ as $K$ in Harris condition. Notice that our proof for Theorem 2 actually suggests that Harris' condition implies quasi-strongly complete continuity. But that is not a surprising result because Yosida and Kakutani have shown that Doeblin condition implies quasi-strongly complete continuity in \cite{Kakutani} and Harris condition can be considered as a generalization of Doeblin condition.\\

In fact, we show in a second proof for Theorem 2 that our condition implies Harris' condition. Since Blume's condition in \cite{Blume} coincides with Doeblin's condition, our condition is a generalization of his condition.\\

\begin{proof} [Alternative Proof]
Since $p(x,K) = q(\pi_i(x),f_x^{-1}(K)) = q(\pi_i(x),E) = 1$ for all $x \in X$, regardless of $i \in \mathbb{N}$, we have $\sup_{x \in X}E_x\tau_K = 1 < \infty$. Now, let $\mu := \delta_{e^*}$. Since $p^{(n)}(x,A) \geq (\inf_{e \in E}q(e,\{e^*\}))^n \geq \epsilon^n$ if $x^* \in A$, we have $p^{(n)}(x,\cdot) \geq \epsilon^n\mu(\cdot)$. Therefore, we have shown that the induced $p$ satisfies the Harris' condition. It is a well-known result from Harris \cite{Harris} that if $p$ satisfies the Harris' condition, there exists a unique invariant measure $\nu$ such that $\int_X p(x,\cdot) \nu(dx) = \nu(\cdot)$ on $\mathcal{X}$. The ergodicity of $p$ follows immediately.
\end{proof}

Finally, one classic characterization of ergodicity is the time average converges to the space average. We state it as a corollary below.

\begin{corollary}
If $F$ and $q$ satisfies the assumptions in Theorem 2,  then the induced $p$ satisfies $\exists \mu^*, \forall f \in M^*, \frac{1}{n} \sum_{i = 0}^{n-1} (T^*)^if(x)$ converges to $\int_X f(x) \mu^*(dx)$ for all $x \in X$.
\end{corollary}

\section{Conclusion}

In this paper, we first show in Lemma 1 that the convergence of $\frac{1}{n}\sum_{i = 1}^nT^i$ to a projection map under the assumption of quasi-weakly complete continuity and hence show the existence of an invariant distribution that depends fully on the initial distribution. Then we show in Theorem 1 that a further assumption of quasi-strong complete continuity leads to the ergodicity up to a multiplicative coefficient in the restriction of $p$ to each ergodic parts on the state space. That is, the influence from a change in the initial distribution on the invariant measure has been limited to a change in the constant coefficients of that invariant measure. Finally, we provide a condition on $F$ and $q$ such that the induced $p$ satisfies the Harris' condition, which not only implies quasi-strongly complete continuity and hence ensures the existence of ergodic decomposition, but also guarantees the uniqueness of the ergodic part and therefore the ergodicity of $p$.\\

In further research on this topic, I would like to provide more verifiable condition on $F$ and $q$ to guarantee the quasi-strongly complete continuity assumption in Theorem 1, since the conclusion in Theorem 1 is already interesting enough due to the very limited dependence of the invariant probability measure on the initial distribution. Also, I would study the applicability of Theorem 2 to real dynamic models in Economics.\\

\bibliography{References}
\bibliographystyle{plain}

\end{document}